\documentclass[a4paper,12pt,intlimits,oneside]{amsart}

\usepackage{enumerate}
\usepackage{amsfonts,amsmath}

\usepackage{latexsym,amssymb}

\newcommand{\comment}[1]{}

\newcounter{rea}
\newcounter{rek}
\newcounter{res}
\begin{document}

\title[On $H^p_w$-boundedness through molecular characterization]{A note on  $H^p_w$-boundedness of Riesz transforms and $\theta$-Calder\'on-Zygmund operators through molecular characterization}         

\author{Luong Dang  KY}    
\address{MAPMO-UMR 6628,
D\'epartement de Math\'ematiques, Universit\'e d'Orleans, 45067
Orl\'eans Cedex 2, France} 
\email{{\tt dangky@math.cnrs.fr}}

\begin{abstract} Let $0<p\leq 1$ and $w$ in the Muckenhoupt class $A_1$. Recently, by using the weighted atomic decomposition and molecular characterization; Lee, Lin and Yang \cite{LLY} (J. Math. Anal. Appl. 301 (2005), 394--400) established that the Riesz transforms $R_j, j=1, 2,...,n$, are bounded on $H^p_w(\mathbb R^n)$. In this note we extend this to the general case of weight $w$ in the  Muckenhoupt class $A_\infty$ through molecular characterization. One difficulty, which has not been taken care in \cite{LLY}, consists in passing from atoms to all functions in $H^p_w(\mathbb R^n)$. Furthermore, the $H^p_w$-boundedness of $\theta$-Calder\'on-Zygmund operators are also given through molecular characterization and atomic decomposition.
\end{abstract}

\maketitle
\newtheorem{theorem}{Theorem}[section]
\newtheorem{lemma}{Lemma}[section]
\newtheorem{proposition}{Proposition}[section]
\newtheorem{remark}{Remark}[section]
\newtheorem{corollary}{Corollary}[section]
\newtheorem{definition}{Definition}[section]
\newtheorem{example}{Example}[section]
\numberwithin{equation}{section}
\newtheorem{Theorem}{Theorem}[section]
\newtheorem{Lemma}{Lemma}[section]
\newtheorem{Proposition}{Proposition}[section]
\newtheorem{Remark}{Remark}[section]
\newtheorem{Corollary}{Corollary}[section]
\newtheorem{Definition}{Definition}[section]
\newtheorem{Example}{Example}[section]

\section{Introduction}

Calder\'on-Zygmund operators and their generalizations on Euclidean space $\mathbb R^n$ have been extensively studied, see for example \cite{Jo, MC, Ya, QY}. In particular, Yabuta \cite{Ya} introduced certain $\theta$-Calder\'on-Zygmund  operators to facilitate his study of certain classes of pseudo-differential operator. 
\begin{Definition}
Let $\theta$ be a nonnegative nondecreasing function on $(0,\infty)$ satisfying $\int_0^1\frac{\theta(t)}{t}dt<\infty$. A continuous function $K:\mathbb R^n\times \mathbb R^n\setminus\{(x,x):x\in \mathbb R^n\}\to\mathbb C$ is said to be a $\theta$-Calder\'on-Zygmund singular integral kernel if there exists a constant $C>0$ such that
$$|K(x,y)|\leq \frac{C}{|x-y|^n}$$
for all $x\ne y$,
$$ |K(x,y)-K(x',y)|+|K(y,x)-K(y,x')|\leq C\frac{1}{|x-y|^n}\theta\Big(\frac{|x-x'|}{|x-y|}\Big)$$
for all $2|x-x'|\leq |x-y|$.

A linear operator $T:\mathcal S(\mathbb R^n)\to\mathcal S'(\mathbb R^n)$ is said to be a $\theta$-Calder\'on-Zygmund operator if $T$ can be extended to a bounded operator on $L^2(\mathbb R^n)$ and there exists a $\theta$-Calderon-Zygmund singular integral kernel $K$ such that for all $f\in C^\infty_c(\mathbb R^n)$ and all $x\notin$supp $f$, we have
$$Tf(x)=\int_{\mathbb R^n}K(x,y)f(y)dy.$$

\end{Definition}

When $K_j(x,y)= \pi^{-(n+1)/2}\Gamma\Big(\frac{n+1}{2}\Big)\frac{x_j-y_j}{|x-y|^{n+1}}, j=1, 2,...,n$, then they are the classical {\sl Riesz transforms}  denoted by $R_j$. 

It is well-known that the Riesz transforms $R_j, j=1, 2,..., n$, are bounded on unweighted Hardy spaces $H^p(\mathbb R^n)$. There are many different approaches to prove this classical result (see \cite{LLY, Le}).  Recently, by using  the weighted molecular theory (see \cite{LL}) and combined with Garc\'ia-Cuerva's atomic decomposition \cite{Ga} for weighted Hardy spaces $H^p_w(\mathbb R^n)$, the authors in \cite{LLY} established that the Riesz transforms $R_j, j=1, 2,...,n$, are bounded on $H^p_w(\mathbb R^n)$. More precisely, they proved that $\|R_jf\|_{H^p_w}\leq C$ for every $w$-$(p,\infty, ts-1)$-atom where $s,t\in\mathbb N$ satisfy $n/(n+s)< p\leq n/(n+s-1)$ and $((s-1)r_w +n)/(s(r_w-1))$ with $r_w$ is the {\sl critical index of $w$ for the reverse H\"older condition}. Remark that this leaves a gap in the proof. Similar gaps exist in some litteratures, for instance in \cite{LL, QY} when the authors establish $H^p_w$-boundedness of Calder\'on-Zygmund type operators. Indeed, it is now well-known that (see \cite{Bo}) the argument {\sl "the operator $T$ is uniformly bounded in $H^p_w(\mathbb R^n)$ on $w$-$(p,\infty,r)$-atoms, and hence it extends to a bounded operator on $H^p_w(\mathbb R^n)$"} is wrong in general. However,  Meda, Sj\"ogren and Vallarino \cite{MSV} establishes that (in the setting of unweighted Hardy spaces) this is correct if one replaces  $L^\infty$-atoms by $L^q$-atoms with $1<q<\infty$. See also \cite{YZ} for $L^2$-atoms with a different method from \cite{MSV}. Later, the authors in \cite{BLYZ} extended these results to the weighted anisotropic Hardy spaces. More precisely,  it is claimed in \cite{BLYZ} that the operator $T$ can be extended to a bounded operator on $H^p_w(\mathbb R^n)$ if it is uniformly bounded on $w$-$(p,q,r)$-atoms for $q_w< q< \infty, r\geq [n(q_w/p-1)]$ where $q_w$ is the {\sl critical index of $w$}.
 
Motivated by \cite{LLY, LL, QY, Bo, BLYZ}, in this paper, we extend {\sl Theorem 1} in \cite{LLY} to $A_\infty$ weights (see {\sl Theorem 1.1}); {\sl Theorem 4} in \cite{LL} (see {\sl Theorem 1.2}), {\sl Theorem 3} in \cite{QY} (see {\sl Theorem 3.1}) to $\theta$-Calder\'on-Zygmund operators; and fill the gaps of the proofs by using the atomic decomposition and molecular characterization of $H^p_w(\mathbb R^n)$ as in \cite{LLY}.

Throughout the whole paper, $C$ denotes a positive geometric constant which is independent of the main parameters, but may change from line to line.  In $\mathbb R^n$, we denote by $B=B(x,r)$ an open ball with center $x$ and radius $r>0$. For any measurable set $E$, we denote by  $|E|$ its  Lebesgue measure, and by $E^c$ the set $\mathbb R^n\setminus E$.

Let us first  recall some notations, definitions and well-known results.

Let $1\leq p<\infty$. A nonnegative locally integrable function $w$ belongs to the {\sl Muckenhoupt class} $A_p$, say $w\in A_p$, if there exists a positive constant $C$ so that
$$\frac{1}{|B|}\int_B w(x)dx\Big(\frac{1}{|B|}\int_B (w(x))^{-1/(p-1)}dx\Big)^{p-1}\leq C, \quad\mbox{if}\; 1<p<\infty,$$
and 
$$\frac{1}{|B|}\int_B w(x)dx\leq C \mathop{\mbox{ess-inf}}\limits_{x\in B}w(x),\quad{\rm if}\; p=1,$$
for all balls $B$ in $\mathbb R^n$. We say that $w\in A_\infty$ if $w\in A_p$ for some $p\in [1,\infty)$.

It is well known that $w\in A_p$, $1\leq p<\infty$, implies $w\in A_q$ for all $q >p$. Also, if $w\in A_p$, $1<p<\infty$, then $w\in A_q$ for some $q\in [1,p)$. We thus write $q_w := \inf\{p \geq 1: w\in A_p \}$ to denote the critical index of $w$. For a measurable set $E$, we note  $w(E) =\int_E w(x) dx$ its weighted measure.

The following lemma gives a characterization of the class $A_p$, $1\leq p<\infty$. It can be found in \cite{GR}.\\\\
{\bf Lemma A.}
The function $w\in A_p$, $1\leq p<\infty$, if and only if, for all nonnegative functions and all balls $B$,
$$\Big(\frac{1}{|B|}\int_B f(x)dx\Big)^p\leq C\frac{1}{w(B)}\int_B f(x)^p w(x)dx.$$

A close relation to $A_p$ is the reverse H\"older condition. If there exist $r >1$ and a fixed constant $C >0$ such that
$$\Big(\frac{1}{|B|}\int_B w^r(x) dx\Big)^{1/r}\leq C\Big(\frac{1}{|B|}\int_B w(x)dx\Big)\qquad\mbox{for every ball}\; B\subset \mathbb R^n,$$
we say that $w$ satisfies {\sl reverse H\"older condition of order $r$} and write $w\in RH_r$. It is known that
if $w\in RH_r$, $r >1$, then $w\in RH_{r+\varepsilon}$ for some $\varepsilon > 0$. We thus write $r_w := \sup\{r >1: w \in RH_r\}$ to denote the {\sl critical index of $w$ for the reverse H\"older condition}.

The following result provides us the comparison between the Lebesgue measure of a set $E$ and its weighted measure $w(E)$. It also can be found in \cite{GR}.\\\\
{\bf Lemma  B.}
Let $w\in A_p\cap RH_r$, $p\geq 1$ and $r>1$. Then there exist constants $C_1, C_2>0$ such that
$$C_1\Big(\frac{|E|}{|B|}\Big)^{p}\leq \frac{w(E)}{w(B)}\leq C_2\Big(\frac{|E|}{|B|}\Big)^{(r-1)/r},$$
for all cubes $B$ and measurable subsets $E\subset B$.

Given a weight function $w$ on $\mathbb R^n$, as usual we denote by $L^q_w(\mathbb R^n)$ the space of all functions $f$ satisfying $\|f\|_{L^q_w}:=(\int_{\mathbb R^n} |f(x)|^q w(x)dx)^{1/q}<\infty$. When $q=\infty$, $L^\infty_w(\mathbb R^n)$ is  $L^\infty(\mathbb R^n)$ and $\|f\|_{L^\infty_w}=\|f\|_{L^\infty}$. Analogously to the classical Hardy spaces, the {\sl weighted Hardy spaces} $H^p_w(\mathbb R^n), p>0$, can be defined in terms of maximal functions. Namely, let $\phi$ be a function in $\mathcal S(\mathbb R^n)$, the Schwartz space of rapidly decreasing smooth functions, satisfying $\int_{\mathbb R^n}\phi(x)dx=1$. Define
$$\phi_t(x)=t^{-n}\phi(x/t),\quad t>0, x\in\mathbb R^n,$$
and the maximal function $f^*$ by
$$f^*(x)=\sup_{t>0}|f*\phi_t(x)|,\quad x\in \mathbb R^n.$$
Then $H^p_w(\mathbb R^n)$ consists of those tempered distributions $f\in\mathcal S'(\mathbb R^n)$ for which $f^*\in L^p_w(\mathbb R^n)$ with the (quasi-)norm
$$\|f\|_{H^p_w}=\|f^*\|_{L^p_w}.$$

In order to show the $H^p_w$-boundedness of Riesz transforms, we characterize weighted Hardy spaces in terms of atoms and molecules in the following way.\\

{\bf Definition of a weighted atom.} Let $0<p\leq 1\leq q\leq \infty$ and $p\ne q$ such that $w\in A_q$. Let $q_w$ be the critical index of $w$. Set $[\cdot]$ the integer function. For $s\in\mathbb N$ satisfying $s\geq [n(q_w/p-1)]$, a function $a\in L^q_w(\mathbb R^n)$ is called  {\sl $w$-$(p,q,s)$-atom centered at $x_0$} if

(i) supp $a\subset B$ for some ball $B$ centered at $x_0$,

(ii) $\|a\|_{L^q_w}\leq w(B)^{1/q-1/p}$,

(iii) $\int_{\mathbb R^n} a(x)x^\alpha dx=0$ for every multi-index $\alpha$ with $|\alpha|\leq s$.\\
Let $H^{p,q,s}_w(\mathbb R^n)$ denote the space consisting of tempered distributions admitting a decomposition $f=\sum_{j=1}^\infty \lambda_j a_j$ in $\mathcal S'(\mathbb R^n)$, where $a_j$'s are $w$-$(p,q,s)$-atoms and $\sum_{j=1}^\infty |\lambda_j|^p<\infty$. And for every $f\in H^{p,q,s}_w(\mathbb R^n)$, we consider the (quasi-)norm
$$\|f\|_{H^{p,q,s}_w}=\inf\Big\{\Big(\sum_{j=1}^\infty |a_j|^p\Big)^{1/p}: f\mathop{=}\limits^{\mathcal S'}\sum_{j=1}^\infty \lambda_j a_j,\;\; \{a_j\}_{j=1}^\infty\;\mbox{are}\; w\mbox{-}(p,q,s)\mbox{-atoms}\Big \}.$$
Denote by $H^{p,q,s}_{w, \rm fin}(\mathbb R^n)$ the {\sl vector space of all finite linear combinations of $w$-$(p,q,s)$-atoms}, and the  {\sl (quasi-)norm} of $f$ in $H^{p,q,s}_{w, \rm fin}(\mathbb R^n)$ is defined by
$$\|f\|_{H^{p,q,s}_{w, \rm fin}}:=\inf\Big\{\Big(\sum_{j=1}^k |\lambda_j|^p\Big)^{1/p}: f= \sum_{j=1}^k \lambda_j a_j, k\in\mathbb N, \{a_j\}_{j=1}^k \;\mbox{are}\; w\mbox{-}(p,q,s)\mbox{-atoms} \Big\}.$$

We have the following atomic decomposition for $H^p_w(\mathbb R^n)$. It can be found in \cite{Ga} (see also \cite{BLYZ, Ky}).\\\\
{\bf Theorem  A.} If the triplet $(p,q,s)$ satisfies the conditions of $w$-$(p,q,s)$-atoms, then $H^p_w(\mathbb R^n)= H^{p,q,s}_w(\mathbb R^n)$ with equivalent norms.

The molecules corresponding to the atoms mentioned above can be defined as follows.\\

{\bf Definition of a weighted molecule.} For $0<p\leq 1\leq q\leq \infty$ and $p\ne q$, let $w\in A_q$ with critical index $q_w$ and critical index $r_w$ for the reverse H\"older condition. Set $s\geq [n(q_w/p-1)]$, $\varepsilon>\max\{sr_w(r_w-1)^{-1}n^{-1}+ (r_w-1)^{-1}, 1/p-1\}$, $a=1-1/p+\varepsilon$, and $b=1-1/q+\varepsilon$. A  {\sl $w$-$(p,q,s,\varepsilon)$-molecule centered at $x_0$} is a function $M\in L^q_w(\mathbb R^n)$ satisfying

(i) $M. w(B(x_0,\cdot-x_0))^b \in L^q_w(\mathbb R^n)$,

(ii) $\|M\|_{L^q_w}^{a/b}\|M.w(B(x_0,\cdot-x_0))^b\|_{L^q_w}^{1-a/b}\equiv \mathfrak N_w(M)<\infty$,

(iii) $\int_{\mathbb R^n} M(x)x^\alpha dx=0$ for every multi-index $\alpha$ with $|\alpha|\leq s$.

The above quantity $\mathfrak N_w(M)$ is called the  $w$-molecular norm of $M$.

In \cite{LL}, Lee and Lin proved that every weighted molecule belongs to the weighted Hardy space $H^p_w(\mathbb R^n)$, and the embedding is continuous.\\\\
{\bf Theorem B.}  Let $0<p\leq 1\leq q\leq \infty$ and $p\ne q$, $w\in A_q$, and $(p,q,s,\varepsilon)$ be the quadruple in the definition of molecule. Then, every $w$-$(p,q,s,\varepsilon)$-molecule $M$ centered at any point in $\mathbb R^n$ is in $H^p_w(\mathbb R^n)$, and $\|M\|_{H^p_w}\leq C \mathfrak N_w(M)$ where the constant $C$ is independent of the molecule.

Although, in general, one cannot conclude that an operator $T$ is bounded on  $H^p_w(\mathbb R^n)$ by checking that their norms have uniform bound on all of the corresponding $w$-$(p,\infty,s)$-atoms (cf. \cite{Bo}). However, this is correct when dealing with $w$-$(p,q,s)$-atoms with $q_w<q<\infty$. Indeed, we have the following result (see [2, Theorem 7.2]).\\\\
{\bf Theorem C.} Let $0<p\leq 1$, $w\in A_\infty$, $q\in (q_w,\infty)$ and $s\in\mathbb Z$ satisfying $s\geq [n(q_w/p-1)]$. Suppose that $T: H^{p,q,s}_{w,\rm fin}(\mathbb R^n)\to H^p_w(\mathbb R^n)$ is a linear operator satisfying
$$\sup\{\|Ta\|_{H^p_w}: a \;\mbox {is any}\; w{\rm -}(p,q,s){\rm -atom}\}<\infty.$$
Then $T$ can be extended to a bounded linear operator on $H^p_w(\mathbb R^n)$.\\

Our  first  main result, which generalizes {\sl Theorem 1} in \cite{LLY}, is as follows:

\begin{Theorem}\label{Riesz}
Let $0<p\leq 1$ and $w\in A_\infty$. Then, the Riesz transforms are bounded on $H^p_w(\mathbb R^n)$.
\end{Theorem}

For the next result, we need the notion $T^*1=0$.
\begin{Definition}
Let $T$ be a $\theta$-Calder\'on-Zygmund operator. We say that $T^*1=0$ if $\int_{\mathbb R^n}Tf(x)dx=0$ for all $f\in L^q(\mathbb R^n), 1< q\leq \infty$, with compact support and $\int_{\mathbb R^n}f(x)dx=0$.
\end{Definition}

We now can give the $H^p_w$-boundedness of $\theta$-Calder\'on-Zygmund type operators, which generalizes {\sl Theorem 4} in \cite{LL} by taking $q=1$ and $\theta(t)=t^\delta$, as follows:

\begin{Theorem}\label{CZO molecule}
Given $\delta\in (0,1]$, $n/(n+\delta)<p\leq 1$, and $w\in A_q\cap RH_r$ with $1\leq q<p(n+\delta)/n, (n+\delta)/(n+\delta- nq)<r$. Let  $\theta$ be a nonnegative nondecreasing function on $(0,\infty)$  with $\int_0^1\frac{\theta(t)}{t^{1+\delta}}dt<\infty$, and  $T$ be a $\theta$-Calder\'on-Zygmund operator satisfying $T^*1=0$. Then $T$ is bounded on $H^p_w(\mathbb R^n)$.
\end{Theorem}

\section{Proof of Theorem \ref{Riesz}}

In order to prove the main theorems, we need the following lemma (see [6, page 412]).\\\\
{\bf Lemma C.} Let $w\in A_r, r>1$. Then there exists a constant $C>0$ such that
$$\int_{B^c}\frac{1}{|x-x_0|^{nr}}w(x)dx\leq C \frac{1}{\sigma^{nr}}w(B)$$
for all balls $B=B(x_0,\sigma)$ in $\mathbb R^n$.

\begin{proof}[Proof of Theorem \ref{Riesz}]
 For $q=2(q_w+1)\in (q_w,\infty)$, then $s:=[n(q/p-1)]\geq [n(q_w/p-1)]$. We now choose (and fix) a positive number $\varepsilon$ satisfying
\begin{equation}\label{molecule}
\max\{sr_w(r_w-1)^{-1}n^{-1}+(r_w-1)^{-1},q/p-1\}<\varepsilon< t(s+1)(nq)^{-1}+q^{-1}-1,
\end{equation}
for some $t\in \mathbb N, t\geq 1$ and $\max\{sr_w(r_w-1)^{-1}n^{-1}+(r_w-1)^{-1},q/p-1\}< t(s+1)(nq)^{-1}+q^{-1}-1$.

Clearly, $\ell:=t(s+1)-1\geq s\geq [n(q_w/p-1)]$. Hence, by Theorem B and Theorem C, it is sufficient to show that for every $w$-$(p,q,\ell)$-atom $f$ centered at $x_0$ and supported in ball $B=B(x_0,\sigma)$, the Riesz transforms $R_j f=K_j*f$, $j=1,2,...,n,$ are $w$-$(p,q,s,\varepsilon)$-molecules with the norm $\mathfrak N_w(R_jf)\leq C$.

Indeed, as $w\in A_q$ by $q=2(q_w+1)\in (q_w,\infty)$. It follows from $L^q_w$-boundedness of Riesz transforms that
\begin{equation}\label{Riesz transforms 1}
\|R_jf\|_{L^q_w}\leq \|R_j\|_{L^q_w\to L^q_w}\|f\|_{L^q_w}\leq C w(B)^{1/q-1/p}.
\end{equation}

To estimate $\|R_jf. w(B(x_0,|\cdot-x_0|))^{b}\|_{L^q_w}$ where $b=1-1/q+\varepsilon$, we write
\begin{eqnarray*}
\|R_jf. w(B(x_0,\cdot-x_0))^{b}\|_{L^q_w}^q
&=&
\int_{|x-x_0|\leq 2\sqrt n\sigma}|R_jf(x)|^q w(B(x_0,|x-x_0|))^{bq}w(x)dx+\\
&&
+\int_{|x-x_0|> 2\sqrt n\sigma}|R_jf(x)|^q w(B(x_0,|x-x_0|))^{bq}w(x)dx\\
&=& I+II.
\end{eqnarray*}
 By Lemma B, we have the following estimate,
\begin{eqnarray*}
I
&=&
\int_{|x-x_0|\leq 2\sqrt n\sigma}|R_jf(x)|^q w(B(x_0,|x-x_0|))^{bq}w(x)dx\\
&\leq&
w(B(x_0, 2\sqrt n\sigma))^{bq}\int_{|x-x_0|\leq 2\sqrt n\sigma}|R_jf(x)|^q w(x)dx\\
&\leq&
C w(B)^{bq} \|R_j\|^q_{L^q_w\to L^q_w}\|f\|^q_{L^q_w}\leq C w(B)^{(b+1/q-1/p)q}.
\end{eqnarray*}
To estimate II, as $f$ is $w$-$(p,q,\ell)$-atom, by the Taylor's fomular and Lemma A, we get
\begin{eqnarray*}
|K_j*f(x)|
&=&
\Big|\int_{|y-x_0|\leq \sigma}\Big(K_j(x-y)-\sum_{|\alpha|\leq \ell}\frac{1}{\alpha!}D^\alpha K_j(x-x_0)(x_0-y)^\alpha\Big)f(y)dy\Big|\\
&\leq&
C\int_{|y-x_0|\leq \sigma}\frac{\sigma^{\ell+1}}{|x-x_0|^{n+\ell+1}}|f(y)|dy\\
&\leq&
C\frac{\sigma^{n+\ell+1}}{|x-x_0|^{n+\ell+1}}w(B)^{-1/q}\|f\|_{L^q_w},
\end{eqnarray*}
for all $x\in (B(x_0, 2\sqrt n\sigma))^c$. As $b=1-1/q+\varepsilon$, it follows from (\ref{molecule}) that $(n+\ell+1)q - q^2nb > nq$. Therefore, by combining the above inequality, Lemma B and Lemma C, we obtain
\begin{eqnarray*}
II
&=&
\int_{|x-x_0|> 2\sqrt n\sigma}|R_jf(x)|^q w(B(x_0,|x-x_0|))^{bq}w(x)dx\\
&\leq&
C \sigma^{(n+\ell+1)q}w(B)^{-1}\|f\|_{L^q_w}^q\int_{|x-x_0|>2\sqrt n\sigma}\frac{1}{|x-x_0|^{(n+\ell+1)q}}w(B(x_0,|x-x_0|))^{bq}w(x)dx\\
&\leq&
C \sigma^{(n+\ell+1)q-q^2nb} w(B)^{(b-1/p)q}\int_{|x-x_0|>2\sqrt n\sigma}\frac{1}{|x-x_0|^{(n+\ell+1)q-q^2nb}}w(x)dx\\
&\leq&
C w(B)^{(b+1/q-1/p)q}.
\end{eqnarray*}

 Thus,
\begin{equation}\label{Riesz transforms 2}
\|R_jf. w(B(x_0,|\cdot-x_0|))^{b}\|_{L^q_w}=(I+II)^{1/q}\leq C w(B)^{b+1/q-1/p}. 
\end{equation}

Remark that $a=1-1/p +\varepsilon$. Combining (\ref{Riesz transforms 1}) and (\ref{Riesz transforms 2}), we obtain 
$$\mathfrak N_w(R_jf)\leq C w(B)^{(1/q-1/p)a/b}  w(B)^{(b+1/q-1/p)(1-a/b)}\leq C.$$

The proof will be concluded if we establish the vanishing moment conditions of $R_jf$. One first consider the following lemma.\\
$\bf{Lemma.}$ For every classical atom $(p,2,\ell)$-atom $g$ centered at $x_0$, we have
$$\int_{\mathbb R^n}R_j g(x)x^\alpha dx=0\quad\mbox{for}\;\; 0\leq |\alpha|\leq s, 1\leq j\leq n.$$
{\sl Proof of the Lemma.} Since $b=1-1/q+\varepsilon<(\ell+1)(nq)^{-1}<(\ell+1)n^{-1}$, we obtain $2(n+\ell+1)-2nb>n$. It is similar to the previous argument, we also obtain that $R_jg$ and $R_j g. |\cdot-x_0|^{nb}$ belong to $L^2(\mathbb R^n)$. Now, we establish that $R_jg.(\cdot-x_0)^\alpha\in L^1(\mathbb R^n)$ for every multi-index $\alpha$ with $|\alpha|\leq s$. Indeed, since $\varepsilon> q/p-1$  by (\ref{molecule}), implies that $2(s-nb)<(s-nb)q'<-n$ by $q=2(q_w+1)>2$, where $1/q+1/q'=1$. We use Schwartz inequality to get
\begin{eqnarray*}
\int_{B(x_0,1)^c}|R_j g(x)(x-x_0)^\alpha|dx
&\leq&
\int_{B(x_0,1)^c}|R_jg(x)||x-x_0|^{s}dx\\
&\leq&
\Big(\int_{B(x_0,1)^c}|R_j g(x)|^2|x-x_0|^{2nb}dx\Big)^{1/2}\Big(\int_{B(x_0,1)^c}|x-x_0|^{2(s-nb)}dx\Big)^{1/2}\\
&\leq&
C \|R_j g.|\cdot-x_0|^{nb}\|_{L^2}<\infty,
\end{eqnarray*}
and
$$\int_{B(x_0,1)}|R_j g(x)(x-x_0)^\alpha|dx\leq |B(x_0,1)|^{1/2}\Big(\int_{B(x_0,1)}|R_j g(x)|^2dx\Big)^{1/2}<\infty.$$

Thus, $R_j g.(\cdot-x_0)^\alpha\in L^1(\mathbb  R^n)$ for any $|\alpha|\leq s$. Deduce that $R_jg(x)x^\alpha \in L^1(\mathbb R^n)$ for any $|\alpha|\leq s$. Therefore,
$$(R_j g(x)x^\alpha)\hat(\xi)=C_\alpha. D^\alpha\widehat{(R_jg)}(\xi)$$
is continuous, with $|C_\alpha|\leq C_s$ ($C_s$ depends only on $s$) for any $|\alpha|\leq s$, where $\hat h$ is used to denote the fourier transform of $h$. Consequently,
$$\int_{\mathbb R^n}R_j g(x)x^\alpha dx= C_\alpha. D^\alpha\widehat{(R_jg)}(0)=C_\alpha. D^\alpha(m_j\hat g)(0),$$
where $m_j(x)=-ix_j/|x|$. Moreover, since $g$ is a classical $(p,2,\ell)$-atom, it follows from [17, Lemma 9.1] that $\hat g$ is $\ell$th order differentiable and $\hat g(\xi)=O(|\xi|^{\ell+1})$ as $\xi\to 0$. We write $e_j$ to be the $j$th standard basis vector of $\mathbb R^n$, $\alpha=(\alpha_1,...,\alpha_n)$ a multi-index of nonnegative integers $\alpha_j$, $\Delta_{he_j}\phi(x)=\phi(x)-\phi(x-he_j)$, $\Delta^{\alpha_j}_{he_j}\phi(x)=\Delta_{he_j}^{\alpha_j-1}\phi(x)-\Delta_{he_j}^{\alpha_j-1}\phi(x-he_j)$ for $\alpha_j\geq 2$, $\Delta^0_{he_j}\phi(x)=\phi(x)$, and $\Delta^\alpha_{h}=\Delta^{\alpha_1}_{he_1}...\Delta^{\alpha_n}_{he_n}$. Then, the boundedness of $m_j$, and $|C_\alpha|\leq C_s$ for $|\alpha|\leq s$, implies
\begin{eqnarray*}
\Big|\int_{\mathbb R^n}R_j g(x)x^\alpha dx\Big|
&=&
|C_\alpha|\Big|\lim\limits_{h\to 0}|h|^{-|\alpha|}\Delta^\alpha_h(m_j\hat g)(0)\Big|\\
&\leq&
C \lim\limits_{h\to 0}|h|^{\ell+1-|\alpha|}=0,
\end{eqnarray*}
for $|\alpha|\leq s$ by $s\leq \ell$. Thus, for any $j=1,2,...,n$, and $|\alpha|\leq s$,
$$\int_{\mathbb R^n}R_j g(x)x^\alpha dx=0.$$
This complete the proof of the lemma.

Let us come back to the proof of  Theorem \ref{Riesz}. As $q/2= q_w+1> q_w$, by Lemma A,
$$\Big(\frac{1}{|B|}\int_B |f(x)|^2dx\Big)^{q/2}\leq C \frac{1}{w(B)}\int_B |f(x)|^q w(x)dx.$$

Therefore, $g:=C^{-1/q}|B|^{-1/p}w(B)^{1/p}f$ is a classical $(p,2,\ell)$-atom since $f$ is $w$-$(p,q,\ell)$-atom associated with ball $B$. Consequently, by the above lemma,
$$\int_{\mathbb R^n}R_j f(x)x^\alpha dx= C^{1/q}|B|^{1/p}w(B)^{-1/p}\int_{\mathbb R^n}R_j g(x)x^\alpha dx=0$$
for all $ j=1,2,...,n$ and $|\alpha|\leq s$. Thus, the theorem  is proved. 
\end{proof}

Following a similar but easier argument, we also have the following $H^p_w$-boundedness of Hilbert transform. We leave details to readers.

\begin{Theorem}
Let $0<p\leq 1$ and $w\in A_\infty$. Then, the Hilbert transform is bounded on $H^p_w(\mathbb R)$.
\end{Theorem}

\section{Proof of theorem \ref{CZO molecule}}

We first consider the following lemma
\begin{Lemma}\label{T1}
Let $p\in(0,1], w\in A_q, 1<q<\infty$, and $T$ be a $\theta$-Calder\'on-Zygmund  operator satisfying $T^*1=0$. Then, $\int_{\mathbb R^n}Tf(x)dx=0$ for all $w$-$(p,q,0)$-atoms $f$.
\end{Lemma}
\begin{proof}[Proof of Lemma \ref{T1}]
 Let $f$ be an arbitrary $w$-$(p,q,0)$-atom associated with ball $B$. It is well-known that there exists $1<r<q$ such that $w\in A_r$. Therefore, it follows from Lemma A that
$$\int_B |f(x)|^{q/r} dx\leq C |B|w(B)^{1/r}\|f\|_{L^q_w}^{q/r}<\infty.$$
We deduce that $f$ is a multiple of  classical $(p,q/r, 0)$-atom, and thus the condition $T^*1=0$ implies $\int_{\mathbb R^n}Tf(x)dx=0$. 
\end{proof}

\begin{proof}
[Proof of Theorem \ref{CZO molecule}]
 Because of the hypothesis, without loss of generality we can assume $q>1$. Futhermore, it is clear that $[n(q_w/p-1)]=0$, and there exists a positive constant $\varepsilon$ such that
\begin{equation}\label{tri1}
\max\Big\{\frac{1}{r_w-1}, \frac{1}{p}-1\Big\}< \varepsilon< \frac{n+\delta}{nq}-1.
\end{equation}

Similarly to the arguments in Theorem \ref{Riesz}, it is sufficient to show that, for every $w$-$(p,q, 0)$-atom $f$ centered at $x_0$ and supported in ball $B=B(x_0,\sigma)$, $Tf$ is a $w$-$(p,q,0,\varepsilon)$-molecule with the norm $\mathfrak N_w(Tf)\leq C$. One first observe that $\int_{\mathbb R^n}Tf(x)dx=0$ by Lemma \ref{T1}, and
$$\sum_{k=0}^\infty\theta(2^{-k})2^{knbq}<\infty,$$
where $b=1-1/q+\varepsilon$, by $\int_0^1 \frac{\theta(t)}{t^{1+\delta}}dt<\infty$ and (\ref{tri1}). We deduce that
\begin{equation}\label{tri2}
\sum_{k=0}^\infty\Big(\theta(2^{-k})2^{knbq}\Big)^q<\infty.
\end{equation}

As $w\subset A_q$, $1<q<\infty$, it follows from [18, Theorem 2.4] that
\begin{equation}\label{tri3}
\|Tf\|_{L^q_w}\leq C \|f\|_{L^q_w}\leq C w(B)^{1/q-1/p}.
\end{equation}

To estimate $\|Tf. w(B(x_0,|\cdot-x_0|))^{b}\|_{L^q_w}$, we write
\begin{eqnarray*}
\|Tf. w(B(x_0,\cdot-x_0))^{b}\|_{L^q_w}^q
&=&
\int_{|x-x_0|\leq 2\sigma}|Tf(x)|^q w(B(x_0,|x-x_0|))^{bq}w(x)dx+\\
&+&
\int_{|x-x_0|> 2\sigma}|Tf(x)|^q w(B(x_0,|x-x_0|))^{bq}w(x)dx= I+II.
\end{eqnarray*}
 By Lemma B, we have the following estimate,
\begin{eqnarray*}
I
&=&
\int_{|x-x_0|\leq 2\sigma}|Tf(x)|^q w(B(x_0,|x-x_0|))^{bq}w(x)dx\\
&\leq&
w(B(x_0, 2\sigma))^{bq}\int_{|x-x_0|\leq 2\sigma}|Tf(x)|^q w(x)dx\\
&\leq&
C w(B)^{bq} \|f\|^q_{L^q_w}\leq C w(B)^{(b+1/q-1/p)q}.
\end{eqnarray*}

To estimate $II$, since $f$ is of  mean zero, by Lemma A, we have
\begin{eqnarray*}
|Tf(x)|
&=&
\Big|\int_{|y-x_0|\leq \sigma}(K(x,y)- K(x, x_0))f(y)dy\Big|\\
&\leq&
C \int_{|y-x_0|\leq \sigma} \frac{1}{|x-x_0|^n}\theta\Big(\frac{|y-x_0|}{|x-x_0|}\Big)|f(y)|dy\\
&\leq&
C \frac{\sigma^n}{|x-x_0|^n} \theta\Big(\frac{\sigma}{|x-x_0|}\Big)w(B)^{-1/q}\|f\|_{L^q_w},
\end{eqnarray*}
for all $x\in (B(x_0, 2\sigma))^c$.  Therefore, by combining the above inequality, Lemma B and (\ref{tri2}), we obtain
\begin{eqnarray*}
II
&=&
\int_{|x-x_0|> 2\sigma}|Tf(x)|^q w(B(x_0,|x-x_0|))^{bq}w(x)dx\\
&\leq&
C w(B)^{-1}\|f\|_{L^q_w}^q\int_{|x-x_0|>2\sigma}\frac{\sigma^{nq}}{|x-x_0|^{nq}}\left(\theta\Big(\frac{\sigma}{|x-x_0|}\Big)\right)^q w(B(x_0,|x-x_0|))^{bq}w(x)dx\\
&\leq&
C  w(B)^{-q/p}\sum_{k=1}^\infty \int_{2^k\sigma<|x-x_0|\leq 2^{k+1}\sigma}\frac{\sigma^{nq}}{|x-x_0|^{nq}}\left(\theta\Big(\frac{\sigma}{|x-x_0|}\Big)\right)^q w(B(x_0,|x-x_0|))^{bq}w(x)dx\\
&\leq&
C w(B)^{(b+1/q-1/p)q}\sum_{k=0}^\infty\Big(\theta(2^{-k})2^{knbq}\Big)^q\leq Cw(B)^{(b+1/q-1/p)q}.
\end{eqnarray*}

Thus,
\begin{equation}\label{tri4}
\|Tf. w(B(x_0,|\cdot-x_0|))^{b}\|_{L^q_w}=(I+II)^{1/q}\leq C w(B)^{b+1/q-1/p}. 
\end{equation}

Remark that $a=1-1/p +\varepsilon$. Combining (\ref{tri3}) and (\ref{tri4}), we obtain 
$$\mathfrak N_w(Tf)\leq C w(B)^{(1/q-1/p)a/b}  w(B)^{(b+1/q-1/p)(1-a/b)}\leq C.$$
This finishes the proof.
\end{proof}

It is well-known that the molecular theory of (unweighted) Hardy spaces of Taibleson and Weiss \cite{TW} is one of useful tools to establish boundedness of operators in Hardy spaces (cf. \cite{TW, Lu}). In the setting of Muckenhoupt weight, this theory has been considered by the authors in \cite{LL}, since then, they have been well used to establish boundedness of operators in weighted Hardy spaces (cf. \cite{LL, LLY, DLL}). However in some cases, the {\sl weighted molecular characterization}, which obtained in \cite{LL}, does not give the best possible results. For Calder\'on-Zygmund type operators in Theorem \ref{CZO molecule}, for instance, it  involves assumption on the {\sl critical index of $w$ for the reverse H\"older condition} as the following theorem does not.

\begin{Theorem}\label{not molecule}
Given $\delta\in (0,1]$, $n/(n+\delta)<p\leq 1$, and $w\in A_q$ with $1\leq q<p(n+\delta)/n$. Let  $\theta$ be a nonnegative nondecreasing function on $(0,\infty)$ with  $\int_0^1 \frac{\theta(t)}{t^{1+\delta}}dt<\infty$, and  $T$ be a $\theta$-Calder\'on-Zygmund operator  satisfying $T^*1=0$. Then $T$ is bounded on $H^p_w(\mathbb R^n)$.
\end{Theorem}

The following corollary give the boundedness of the classical Calder\'on-Zygmund type operators on weighted Hardy spaces (see [15, Theorem 3]).
\begin{Corollary}\label{classical}
Let $0<\delta\leq 1$ and $T$ be the classical $\delta$-Calder\'on-Zygmund operator, i.e. $\theta(t)=t^\delta$, satisfying $T^*1=0$. If $n/(n+\delta)<p\leq 1$ and $w\in A_q$ with $1\leq q<p(n+\delta)/n$, then $T$ is bounded on $H^p_w(\mathbb R^n)$.
\end{Corollary}

\begin{proof}[Proof of Corollary \ref{classical}]
 By taking $\delta'\in (0,\delta)$ which is close enough $\delta$. Then, we apply Theorem 3.1 with $\delta'$ instead of $\delta$.
\end{proof}

\begin{proof}
[Proof of Theorem \ref{not molecule}]
 Without loss of generality we can assume $1<q< p(n+\delta)/n$.  Fix $\phi\in \mathcal S(\mathbb R^n)$ with $\int_{\mathbb R^n}\phi(x)dx\ne 0$. By Theorem C, it is sufficient to show that for every $w$-$(p,q, 0)$-atom $f$ centered at $x_0$ and supported in ball $B=B(x_0,\sigma)$, $\|(Tf)^*\|_{L^p_w}\leq C$. In order to do this, one write
\begin{eqnarray*}
\|(Tf)^*\|_{L^p_w}^p&=&\int_{|x-x_0|\leq 4\sigma}\Big((Tf)^*(x)\Big)^pw(x)dx+ \int_{|x-x_0|> 4\sigma}\Big((Tf)^*(x)\Big)^pw(x)dx\\
&=& L_1+ L_2.
\end{eqnarray*}

By H\"older inequality, $L^q_w$-boundedness of the maximal function and Lemma B, we get
\begin{eqnarray*}
L_1
&\leq&
 \left(\int_{|x-x_0|\leq 4\sigma}\Big((Tf)^*(x)\Big)^q w(x)dx\right)^{p/q} \Big(\int_{|x-x_0|\leq 4\sigma}w(x)dx\Big)^{1-p/q}\\
&\leq&
C \|f\|_{L^q_w}^p w(B(x_0,4\sigma))^{1-p/q}\leq C.
\end{eqnarray*}

To estimate $L_2$, we first estimate $(Tf)^*(x)$ for $|x-x_0|> 4\sigma$. For any $t>0$, since $\int_{\mathbb R^n}Tf(x)dx=0$ by Lemma 3.1, we get
 \begin{eqnarray*}
|Tf*\phi_t(x)|&=& \left|\int_{\mathbb R^n} Tf(y)\frac{1}{t^n}\left(\phi\Big(\frac{x-y}{t}\Big)- \phi\Big(\frac{x-x_0}{t}\Big)\right)dy\right|\\
&\leq&
\frac{1}{t^n}\int_{|y-x_0|< 2\sigma}|Tf(y)|\left|\phi\Big(\frac{x-y}{t}\Big)- \phi\Big(\frac{x-x_0}{t}\Big)\right|dy\\
&&
+ \frac{1}{t^n}\int_{2\sigma\leq |y-x_0|< \frac{|x-x_0|}{2}}\cdots +  \frac{1}{t^n}\int_{|y-x_0|\geq \frac{|x-x_0|}{2}}\cdots\\
&=& E_1(t)+ E_2(t)+ E_3(t).
\end{eqnarray*}

As $|x-x_0|> 4\sigma$, by the mean value theorem, Lemma A and Lemma B, we get
\begin{eqnarray*}
E_1(t)&=& \frac{1}{t^n}\int_{|y-x_0|< 2\sigma}|Tf(y)|\left|\phi\Big(\frac{x-y}{t}\Big)- \phi\Big(\frac{x-x_0}{t}\Big)\right|dy\\
&\leq&
\frac{1}{t^n}\int_{|y-x_0|< 2\sigma}|Tf(y)|\frac{|y-x_0|}{t}\sup\limits_{\lambda\in (0,1)}\left|\nabla\phi\Big(\frac{x-x_0+ \lambda (y-x_0)}{t}\Big)\right|dy\\
&\leq&
C\frac{\sigma}{|x-x_0|^{n+1}}\int_{|y-x_0|< 2\sigma}|Tf(y)|dy\\
&\leq&
C \frac{\sigma}{|x-x_0|^{n+1}}|B(x_0, 2\sigma)|w(B(x_0,2\sigma))^{-1/q}\|Tf\|_{L^q_w}\\
&\leq&
C  \frac{\sigma^{n+1}}{|x-x_0|^{n+1}}w(B)^{-1/q}\|f\|_{L^q_w}\leq C \frac{\sigma^{n+1}}{|x-x_0|^{n+1}}w(B)^{-1/p}.
\end{eqnarray*}

Similarly, we also get
\begin{eqnarray*}
E_2(t)&\leq& \frac{1}{t^n}\int_{2\sigma\leq |y-x_0|< \frac{|x-x_0|}{2}}\left|\int_{\mathbb R^n}f(z)\Big(K(y,z)- K(y,x_0)\Big)dz\right|\frac{|y-x_0|}{t}\\
&&\quad\quad\quad\quad\quad\quad\quad\quad\quad\quad\quad\quad\quad\quad  \times\sup\limits_{\lambda\in (0,1)}\left|\nabla\phi\Big(\frac{x-x_0+ \lambda (y-x_0)}{t}\Big)\right|dy\\
&\leq&
C\frac{1}{|x-x_0|^{n+1}}\int_{2\sigma\leq |y-x_0|< \frac{|x-x_0|}{2}}|y-x_0|\int_{|z-x_0|< \sigma}|f(z)|\frac{1}{|y-x_0|^n}\theta\Big(\frac{|z-x_0|}{|y-x_0|}\Big)dzdy\\
&\leq&
C\Big(\frac{\sigma}{|x-x_0|}\Big)^{n+1}\int_{2\sigma/|x-x_0|}^{1/2}\frac{\theta(t)}{t^2}dt w(B)^{-1/p}\\
&\leq&
C \Big(\frac{\sigma}{|x-x_0|}\Big)^{n+1}\Big(\frac{|x-x_0|}{2\sigma}\Big)^{1-\delta}\int_{2\sigma/|x-x_0|}^{1/2}\frac{\theta(t)}{t^{1+\delta}}dt w(B)^{-1/p}\\
&\leq&
 C \Big(\frac{\sigma}{|x-x_0|}\Big)^{n+\delta}w(B)^{-1/p}.
\end{eqnarray*}

 Next, let us look at $L_3$. Similarly, we also have
\begin{eqnarray*}
E_3(t)&\leq& \frac{1}{t^n}\int_{|y-x_0|\geq \frac{|x-x_0|}{2}}\left|\int_{\mathbb R^n}f(z)\Big(K(y,z)- K(y,x_0)\Big)dz\right| \left(\Big|\phi\Big(\frac{y-x_0}{t}\Big)\Big|+ 2\Big|\phi\Big(\frac{x-x_0}{t}\Big)\Big|\right)dy\\
&\leq&
C\frac{1}{|x-x_0|^n}\int_{|y-x_0|\geq \frac{|x-x_0|}{2}}\int_{|z-x_0|< \sigma}|f(z)|\frac{1}{|y-x_0|^n}\theta\Big(\frac{|z-x_0|}{|y-x_0|}\Big)dzdy\\
&\leq&
C \Big(\frac{\sigma}{|x-x_0|}\Big)^n \int_0^{2\sigma/|x-x_0|}\frac{\theta(t)}{t}dtw(B)^{-1/p}\\
&\leq&
 C\Big(\frac{\sigma}{|x-x_0|}\Big)^n \int_0^{2\sigma/|x-x_0|}\frac{\theta(t)}{t^{1+\delta}}dt\Big(\frac{2\sigma}{|x-x_0|}\Big)^\delta w(B)^{-1/p}\\
&\leq&
 C\Big(\frac{\sigma}{|x-x_0|}\Big)^{n+\delta}w(B)^{-1/p}.
\end{eqnarray*}

Therefore, for all $|x-x_0|> 4\sigma$, 
$$(Tf)^*(x)=\sup\limits_{t>0}( E_1(t)+ E_2(t) + E_3(t))\leq C \Big(\frac{\sigma}{|x-x_0|}\Big)^{n+\delta}w(B)^{-1/p}.$$

Combining this, Lemma C and Lemma B, we obtain that
\begin{eqnarray*}
L_2=\int_{|x-x_0|> 4\sigma}\Big((Tf)^*(x)\Big)^pw(x)dx
&\leq&
 C \int_{|x-x_0|> 4\sigma}\frac{\sigma^{(n+\delta)p}}{|x-x_0|^{(n+\delta)p}}w(B)^{-1}w(x)dx\\
&\leq&
C w(B)^{-1} w(B(x_0, 4\sigma))\leq C,
\end{eqnarray*}
since $(n+\delta)p>nq$. This finishes the proof.
\end{proof}


\end{document}